\documentclass[12pt]{amsart}
\usepackage{amssymb}
\usepackage{amsmath}
\usepackage{amsthm}
\usepackage{amsfonts}
\usepackage{bbm}
\usepackage{enumitem}
\usepackage{epsf,epsfig,amsfonts,graphicx}
\usepackage{color}
\usepackage[T1]{fontenc}
\usepackage{a4wide}
\newtheorem{theorem}{Theorem}

\newtheorem{remark}{Remark}
\newtheorem{proposition}{Proposition}
\newtheorem{corollary*}{Corollary}
\newtheorem{example}{Example}

\newcommand{\Aas}{A(s)}

\newcommand{\aff}{{\rm Aff}}

\newcommand{\affn}{{\rm Aff}(n)}

\newcommand{\bas}{b(s)}
\newcommand{\be}{\begin{equation}}
\newcommand{\ee}{\end{equation}}
\newcommand{\ben}{\begin{equation*}}

\newcommand{\een}{\end{equation*}}
\newcommand{\cas}{c_{s}}

\newcommand{\gl}{\textrm{Gl}}

\newcommand{\Identity}{\mathbbm{1}}
\newcommand{\kea}{k_{ea}}
\newcommand{\Kea}{{\rm K}_{ea}}
\newcommand{\ka}{k_{ga}}
\newcommand{\kga}{k_{ga}}

\newcommand{\R}{\mathbb{R}}

\newcommand{\sa}{s_{ga}}
\newcommand{\sea}{s_{ea}}

\newcommand{\sign}{{\rm sign}}

\newcommand{\trace}{{\rm tr}}

\newcommand{\weg}[1]{}
\newcommand{\z}{\mathbb{Z}}

\title{Solitons of the midpoint mapping and affine curvature}
\author{Christine Rademacher}
\address{Technische Hochschule N\"urnberg Georg Simon Ohm,
Fakult\"at Angewandte Mathematik, Physik und 
Allgemeinwissenschaften,
Postfach 210320, 90121 N\"urnberg, Germany}
\email{christine.rademacher@th-nuernberg.de}
\author{Hans-Bert Rademacher}
\address{Universit\"at Leipzig, Mathematisches Institut, 
04081 Leipzig, Germany}
\email{hans-bert.rademacher@math.uni-leipzig.de}
\date{2020-07-28}
\begin{document}

\baselineskip 18pt

\begin{abstract}
	For a polygon $x=(x_j)_{j\in \z}$ in $\R^n$ we consider
	the midpoints polygon  $(M(x))_j=\left(x_j+x_{j+1}\right)/2\,.$
	We call a polygon a \emph{soliton} of
	the midpoints mapping $M$ if
	its midpoints polygon is the image of the polygon
	under an invertible affine map.
	We show that a large class of these polygons lie on an orbit of a 
	one-parameter subgroup of the affine group
	acting on $\R^n.$ These smooth curves
	are also characterized as solutions
	of the differential equation
	$\dot{c}(t)=Bc (t)+d$
	for a matrix $B$ and a vector $d.$
	For $n=2$ these curves 
	are curves of
	constant generalized-affine curvature 
	$\kga=\kga(B)$ depending on $B$ parametrized
	by generalized-affine arc length unless
	they are parametrizations of a 
	parabola, an ellipse,
	or a hyperbola.
\end{abstract}

\maketitle

\section{Introduction}
We consider an infinite polygon $(x_j)_{j \in \z}$ given by its vertices
$x_j \in \R^n$
in an $n$-dimensional real vector space $\R^n$ resp. an $n$-dimensional affine
space $\mathbb{A}^n$ modelled after $\R^n.$  
For a 
parameter $\alpha \in (0,1)$ we introduce the 
polygon $M_{\alpha}(x)$ whose vertices
are given by
\begin{equation*}
 \left(M_{\alpha}(x)\right)_j:=(1-\alpha) x_j+\alpha x_{j+1}\,.
\end{equation*}
For $\alpha=1/2$ this defines the {\em midpoints polygon}
$M(x)=M_{1/2}(x).$
On the space $\mathcal{P}=\mathcal{P}(\R^n)$ of polygons
in $\R^n$ this defines a 
discrete curve shortening 
process $M_{\alpha}: \mathcal{P}\longrightarrow \mathcal{P},$
already considered by Darboux~\cite{Darboux1878}
in the case of a \emph{closed} resp.
\emph{periodic} polygon.
For a discussion of this elementary geometric construction
see Berlekamp et al.~\cite{Berlekamp1965}.

The mapping $M_{\alpha}$ is
invariant under the canonical action
of the affine group.
The {\em affine group} $\aff (n)$
in dimension $n$ is the 
set of {\em affine maps} $(A,b): \R^n \longrightarrow \R^n,
x \longmapsto Ax +b.$ 
Here  $A \in \gl(n)$ is an invertible matrix and
 $b \in \R^n$ a vector. 
 The translations $x \longmapsto x+b$
determined by a vector $b$ form a subgroup isomorphic to $\R^n.$
Let $\alpha \in (0,1).$
We call a polygon $x_j$ a {\em soliton} for the 
process $M_{\alpha}$
(or affinely invariant under $M_{\alpha}$)
if there is an affine map $(A,b)\in \affn$
such that
\begin{equation}
\label{eq:soliton}
\left(M_{\alpha}(x)\right)_j=A x_j +b
\end{equation}
for all $j \in \mathbb{Z}.$
In Theorem~\ref{thm:polygon-soliton} we 
describe these solitons explicitely and
discuss under which assumptions they lie
on the orbit of a one-parameter subgroup
of the affine group acting canonically on $\R^n.$
We call a smooth curve $c:\R \longrightarrow \R^n$
a \emph{soliton} of the mapping $M_{\alpha}$
resp. \emph{invariant} under the mapping $M_{\alpha}$
	if there is for some $\epsilon >0$
	a smooth mapping
	$s \in (-\epsilon,\epsilon)\longmapsto (A(s),b(s))\in 
	\aff (n)$ such that for all 
	$s\in (-\epsilon,\epsilon)$ and $t \in \R:$
	\begin{equation}
	\label{eq:cstilde}
	\tilde{c}_s(t):=(1-\alpha)c(t)+\alpha c(t+s)=
	A(s)c(t)+b(s)\,.
	\end{equation}
	Then for some $t_0 \in \R$ and $s \in (-\epsilon,\epsilon)$ the
	polygon $x_j=c(js+t_0), j\in \mathbb{Z}$	
	is a soliton of $M_{\alpha}.$
	The parabola is an example of a soliton
	of $M=M_{1/2},$
	cf. Figure~\ref{fig:one}
	and Example~\ref{exa:n=2}, Case (e).
	We show in Theorem~\ref{thm:curve-soliton}
	that the smooth curves invariant under 
	$M_{\alpha}$ coincide with the orbits of 
	a one-parameter subgroup of the affine group
	$\aff(n)$ acting canonically on $\R^n.$ 
	For $n=2$ we give a 
	characterization of these curves in terms of
	the \emph{general-affine curvature} in 
	Section~\ref{sec:affine-curvature}. 
		
	The authors discussed 
	{\em solitons,} i.e. curves affinely invariant
	under the curve shortening process 
	$T: \mathcal{P}(\R^n)\longrightarrow 
	\mathcal{P}(\R^n)$
	with
	\begin{equation}
	\label{eq:rademacher1}
	\left( T(x)\right)_j=
	\frac{1}{4}\left\{x_{j-1}+2x_j+x_{j+1}\right\}
	\end{equation}	
	in \cite{Rademacher2017}. 
	The solitons of $M=M_{1/2}$ form a subclass of 
	the solitons of $T,$ since 
	$\left(T(x)\right)_j=
	\left(M^2(x)\right)_{j-1}.$
	Instead of the discrete evolution  
	of polygons one can also investigate the 
	evolution of polygons under a linear 
	flow, cf. Viera and Garcia~\cite{VG}
	and \cite[sec.4]{Rademacher2017}
	or a non-linear flow, cf.
	Glickenstein and Liang~\cite{GL}.
\begin{figure}
		\includegraphics[width=0.6\textwidth]{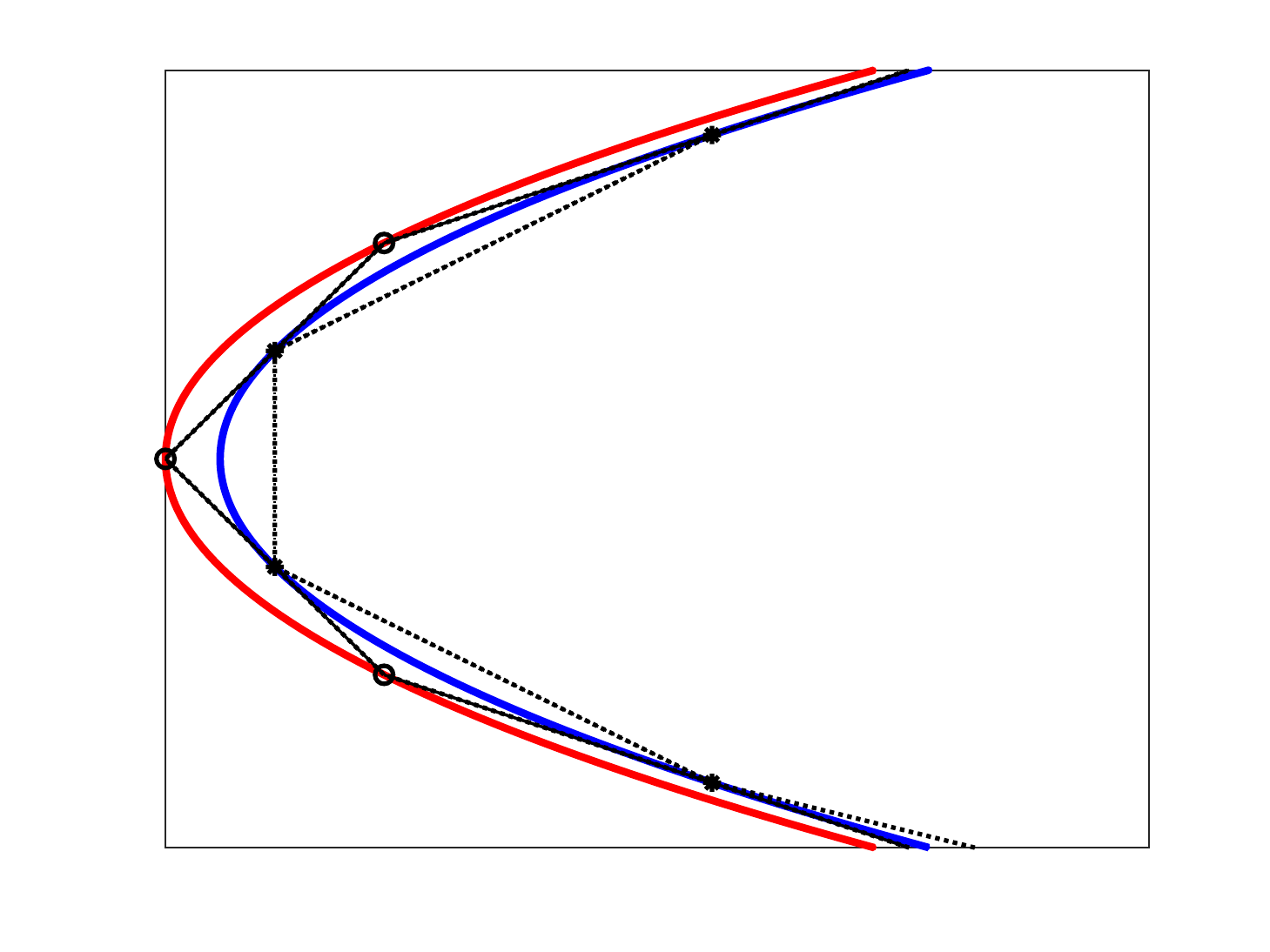}  
		\caption{The parabola $c(t)=(t^2/2 , t)$ as soliton of the midpoints map $M.$}
		\label{fig:one}
\end{figure}
\section{The affine group and systems of linear 
differential equations of first order}
The affine group
$\aff(n)$ is a semidirect product of the general
linear group $\gl(n)$ and the group 
$\mathbb{R}^n$ of translations.
There is a linear representation 
\begin{equation*}
(A,b) \in \aff(n) 
\longrightarrow
\left(
\begin{array}{c|c}
 A & b\\
 \hline 0 &1
\end{array}
\right) \in
\gl(n+1),
\end{equation*}
of the affine group in the general linear group
$\gl(n+1),$
cf.~\cite[Sec.5.1]{Kuehnel2011}.
We use the following identification 
\begin{equation}
\label{eq:identification}
 \left(
\begin{array}{c|c}
 A & b\\
 \hline 0 &1
\end{array}
\right) 
\left(
\begin{array}{c}
 x\\ \hline 1
\end{array}
\right)
=
\left(
\begin{array}{c}
 A x+b\\ \hline 1
\end{array}
\right)\,.
\end{equation}
Hence we can identify the image of a vector
$x \in \R^n$ under the affine map $x \longmapsto
Ax+b$ with the image $\left(
\begin{array}{c}Ax +b\\ \hline 1
\end{array}\right)$
of the extended
vector 
$\left(
\begin{array}{c}x\\ \hline 1\end{array}\right)$.
Using this identification we can write down the solution of 
an inhomogeneous system of linear differential equations
with constant coefficients using the power series
$F_B(t)$ which we introduce now:
\begin{proposition}
\label{pro:fb}
 For a real $(n,n)$-matrix $B \in M_{\R}(n)$ we denote by 
 $F_B(t)\in M_{\R}(n)$ the following power series:
 \begin{equation}
 \label{eq:fb-def}
  F_B(t)=\sum_{k=1}^{\infty}
  \frac{t^k}{k!}B^{k-1}\,.
 \end{equation}
  \begin{itemize}
 	\item[(a)] We obtain for its derivative:
 	\be
 	\label{eq:fb-derivative}
 	\frac{d}{dt} F_B(t)=\exp(Bt)=BF_B(t)+\Identity\,.
 	\ee
 	The function $F_B(t)$ satisfies the following functional
 	equation:
 	\be
 	\label{eq:fb-functional-eq}
 	F_B(t+s)=F_B(s)+\exp(Bs)F_B(t)\,,
 	\ee 
 	resp. for $j \in \mathbb{Z}, j\ge 1:$
 	\begin{eqnarray*}
 	\nonumber
 	F_B(j)=\left\{\Identity+\exp(B)+\exp(2B)+\ldots+
 	\exp((j-1)B)\right\}F_B(1)\\
 	\label{eq:fb-functional-eq-j}
 	=\left(\exp(B)-\Identity\right)^{-1}
 	\left(\exp(jB)-\Identity\right)F_B(1)\,.
 	\end{eqnarray*}
 	
 	\item[(b)] 
 The solution $c(t)$ of the 
 inhomogeneous system of linear differential equations
 \begin{equation}
  \label{eq:diffeq}
 \dot{c}(t)=B c(t)+d
 \end{equation}
 with constant coefficients 
 (i.e. $B\in M_{\R}(n,n),d\in \R^n$)
 and
 with initial condition $v=c(0)$
 is given by:
 \begin{equation}
 \label{eq:ct}
  c(t)=v+F_B(t)\left(Bv+d\right)
  =\exp(Bt)(v)+F_B(t)(d).  
 \end{equation}
\end{itemize}
 \end{proposition}
 \begin{proof}
 	(a)
 	Equation~\eqref{eq:fb-derivative} follows immediately
 	from Equation~\eqref{eq:fb-def}.
 	Then we compute
 	\begin{eqnarray*}
 		\frac{d}{dt} \left(F_B(t+s)-
 		\exp(Bs)F_B(t)\right)
 		=\exp(B(t+s))-\exp(Bs)\exp(Bt)=0\,.
 	\end{eqnarray*}
 	Since $F_B(0)=0$ Equation~\eqref{eq:fb-functional-eq}
 	follows. And this implies 
 	Equation~\eqref{eq:fb-functional-eq-j}.
 	
 	\medskip
 	
(b)  We can write the solution of the 
  differential equation~\eqref{eq:diffeq}
  \begin{equation*}
  \frac{d}{dt}
	\label{eq:eq:diffeq1}
   \left(
   \begin{array}{c}
    c(t)\\ \hline 1
   \end{array}
   \right)
   =
   \left(
   \begin{array}{c|c}
    B&d\\ \hline 0&0
   \end{array}
   \right)
   \left(
   \begin{array}{c}
    c(t)\\ \hline 1
   \end{array}
   \right)
   \end{equation*}
   as follows:
    \begin{eqnarray}
    \label{eq:ct-expB}
   \left(
   \begin{array}{c}
    c(t)\\ \hline
    1   
\end{array}
\right)=
\exp \left(
\left(
\begin{array}{c|c}
  B& d\\ \hline
  0&0
\end{array}
\right)
t 
\right)
\left(
\begin{array}{c}
 v \\ \hline 1
 \end{array}
 \right) 
  =\\
  \nonumber
\left(
\begin{array}{c|c}
  \exp(Bt)& F_B(t)(d)\\
  \hline
  0&1
\end{array}
\right)  
\left(
\begin{array}{c}
 v \\ \hline 1
 \end{array}
 \right)
 =
 \left(
 \begin{array}{c}
  \exp(Bt)(v)+F_B(t)(d) \\ \hline 1
  \end{array}
  \right) 
\end{eqnarray}
which is Equation~\eqref{eq:ct}. One could also
differentiate Equation~\eqref{eq:ct}
and use Equation~\eqref{eq:fb-derivative}
 \end{proof}
 \begin{remark}
 	\label{rem:one-parameter-subgroup}
 \rm
   Equation~\eqref{eq:cstilde}
 shows that $c(t)$ is the orbit 
  \begin{equation*}
 t \in \R\longmapsto
 c(t)=\exp \left(
 \left(
 \begin{array}{c|c}
 B& d\\ \hline
 0&0
 \end{array}
 \right)
 t 
 \right)
 \left(
 \begin{array}{c}
 v \\ \hline 1
 \end{array}
 \right) 
 \in \R^n\,.
 \end{equation*} 
 of
 the one-parameter subgroup
 \begin{equation*}
 t \in \R\longmapsto
 \exp \left(
 \left(
 \begin{array}{c|c}
 	B& d\\ \hline
 	0&0
 \end{array}
 \right)
 t 
 \right)\in \affn
 \end{equation*}
 of the affine group
 $\affn$ acting canonically on
 $\R^n.$	
 
 \end{remark}
\section{Polygons invariant under $M_{\alpha}$}
\begin{theorem}
	\label{thm:polygon-soliton}
 Let $(A,b): x \in \R^n \longmapsto
 Ax + b\in \R^n$ be an affine map and $v \in \R^n.$
 Assume that for $\alpha \in (0,1)$ the value
 $1-\alpha$ is not an eigenvalue of $A,$ i.e.
 the matrix $A_{\alpha}:=
 \alpha^{-1}\left(A+(\alpha-1)\Identity\right) $ is invertible.
 Then the following statements hold:
 
 \smallskip
 
 (a) There is a unique polygon $x \in \mathcal{P}(\R^n)$ with
 $x_0=v$ which is a soliton for 
 $M_{\alpha}$ resp. affinely invariant under 
 the mapping $M_{\alpha}$ with respect to the affine
 map $(A,b),$ cf. Equation~\eqref{eq:soliton}.
 If $b_{\alpha}=\alpha^{-1}b,$ then for 
 $j>0:$
 \begin{eqnarray}
 \nonumber
 x_j&=&A_{\alpha}^j(v)+
 A_{\alpha}^{j-1}\left(b_{\alpha}\right)+\ldots+
 A_{\alpha}\left(b_{\alpha}\right)+
 b_{\alpha}\\
 \label{eq:xj}
 &=&v+\left(A_{\alpha}^j-\Identity\right)
 \left( v +
 \left(A_{\alpha}-\Identity\right)^{-1}
 \left(b_{\alpha}\right)
 \right)\,.
 \end{eqnarray}
 and for $j<0:$
 \begin{eqnarray}
 \nonumber
 x_j&=&A_{\alpha}^j(v)-
 A_{\alpha}^{j}\left(b_{\alpha}\right)+\ldots+
 A_{\alpha}^{-1}\left(b_{\alpha}\right)\\
 \label{eq:xj-negative-j}
 &=&v+\left(A_{\alpha}^j-\Identity\right)
 \left( v -
 \left(A_{\alpha}^{-1}-\Identity\right)^{-1}
 \left(
 A_{\alpha}^{-1} (b_{\alpha}
 \right)
 \right)\,.
 \end{eqnarray}

 \smallskip
 
 (b) If $A_{\alpha}=\exp\left(B_{\alpha}\right)$
 for a $(n,n)$-matrix $B_{\alpha}$ and if
 $b_{\alpha}=F_{B_{\alpha}}(1)\left(d_{\alpha}\right)$
 for a vector $d_{\alpha}\in \R^n$ then
 the polygon $x_j$ lies on the smooth
 curve 
 \begin{equation*}
 \label{eq:ctB}
 c(t)=v+
 F_{B_{\alpha}}(t)\left(B_{\alpha}v+d_{\alpha}\right)
 \end{equation*}
 i.e. $x_j=c(j)$ for all $j \in \z.$ 
 \end{theorem}
\begin{proof}
(a)
By Equation~\eqref{eq:soliton} we have
\begin{equation*}
 (1-\alpha) x_j +\alpha x_{j+1}=A x_j+b
\end{equation*}
for all $j \in \z.$ Hence the polygon is given by $x_0=v$ and
the recursion formulae
\begin{equation*}
x_{j+1}=A_{\alpha}(x_j)+b_{\alpha}\,;\,
x_{j}=A_{\alpha}^{-1}
\left( x_{j+1}-b_{\alpha}\right)\,.
\end{equation*}
for all $j \in \z.$
Then Equation~\eqref{eq:xj} 
and Equation~\eqref{eq:xj-negative-j}
follow.

\medskip

(b)
For $A_{\alpha}=\exp B_{\alpha}; b_{\alpha}=F_{B_{\alpha}}(d_{\alpha})$
we obtain 
from Equation~\eqref{eq:fb-derivative}
for all $j \in \mathbb{Z}:$
$A_{\alpha}-\Identity=B_{\alpha}F_{B_{\alpha}}(1)$
and 
$A_{\alpha}^j-\Identity=B_{\alpha}F_{B_{\alpha}}(j).$
Hence for $j>0:$
\begin{eqnarray*}
x_j
&=& v+
\left(
A_{\alpha}^j-\Identity
\right) 
\left( 
v+
\left( 
A_{\alpha}-\Identity
\right)^{-1} b_{\alpha}
\right)
\\
&=& v+B_{\alpha} F_{B_{\alpha}}(j)
\left( 
v+\left( 
B_{\alpha} F_{B_{\alpha}}(1)
\right)^{-1}
(b_{\alpha})
\right)	\\
&=& v+F_{B_{\alpha}}(j)
\left( 
B_{\alpha} v+d_{\alpha}
\right)=c(j)\,.
\end{eqnarray*}
The Functional 
Equation~\eqref{eq:fb-functional-eq}
for $F_B(t)$ implies
$0=F_B(0)=F_B(-1+1)=F_B(-1)+\exp(-B)F_B(1),$
hence
\ben
F_B(-1)=-\exp(-B)F_B(1)\,;\, F_B(-1)^{-1}=-\exp(B)F_B(1)^{-1}.
\een
Note that the matrices $B, F_B(t), F_B(t)^{-1}$ commute.
With this identity we obtain for $j<0:$ 
\begin{eqnarray*}
	x_j
	&=& v+
	\left(
	A_{\alpha}^j-\Identity
	\right) 
	\left( 
	v-
	\left( 
	A_{\alpha}^{-1}-\Identity
	\right)^{-1}
	\left(  
	A_{\alpha}^{-1} b_{\alpha}
	\right)
	\right)
	\\
	&=& v+B_{\alpha} F_{B_{\alpha}}(j)
	\left( 
	v-\left( 
	B_{\alpha} F_{B_{\alpha}}(-1)
	\right)^{-1}
	\exp(-B_{\alpha})(b_{\alpha})
	\right)	\\
		&=& v+B_{\alpha} F_{B_{\alpha}}(j)
	\left( 
	v- 
	 F_{B_{\alpha}}(-1)^{-1}
	 B_{\alpha}^{-1}
	\exp(-B_{\alpha})
	F_{B_{\alpha}}(1)(d_{\alpha})
	\right)	\\	
	&=& v+F_{B_{\alpha}}(j)
	\left( 
	B_{\alpha} v+d_{\alpha}
	\right)=c(j)\,.
\end{eqnarray*}
\end{proof}
\begin{remark}
(a)
 \rm
 Using the identification 
 Equation~\eqref{eq:identification}
 we can write 
 \begin{equation}
  \left(
  \begin{array}{c}
  x_{j+1} \\ \hline 1
  \end{array}
  \right)
  =
  \left(
  \begin{array}{c|c}
     A_{\alpha}& b_{\alpha}\\
   \hline
   0&1
  \end{array}
  \right)
  \left(
  \begin{array}{c}
   x_j\\ \hline 1
  \end{array}
  \right)
 \,;\,
    \left(
   \begin{array}{c}
   x_{j} \\ \hline 1
   \end{array}
   \right)
   =
   \left(
   \begin{array}{c|c}
    A_{\alpha}& b_{\alpha}\\
    \hline
    0&1
   \end{array}
   \right)^j
   \left(
   \begin{array}{c}
    v\\ \hline 1
   \end{array}
   \right)
  \end{equation}
  for all $j\in \z.$
  
  \smallskip
  
  (b) If $A_{\alpha}=\exp\left(B_{\alpha}\right)$
  for a $(n,n)$-matrix $B_{\alpha}$ and if
  $b_{\alpha}=F_{B_{\alpha}}(1)\left(d_{\alpha}\right)$
  for a vector $d_{\alpha}\in \R^n$ then
  we obtain from
  Equation~\eqref{eq:ct-expB}:
       \begin{eqnarray*}
  	\label{eq:one-parameter}
  	\left(
  	\begin{array}{c}
  		c(t)\\ \hline
  		1   
  	\end{array}
  	\right)=
  	\exp \left(
  	\left(
  	\begin{array}{c|c}
  		B_{\alpha}& d_{\alpha}\\ \hline
  		0&0
  	\end{array}
  	\right)
  	t 
  	\right)
  	\left(
  	\begin{array}{c}
  		v \\ \hline 1
  	\end{array}
  	\right) 
  	=
  	\left(
  	\begin{array}{c|c}
  		\exp(B_{\alpha}t)& F_{B_{\alpha}}(t)(d_{\alpha})\\
  		\hline
  		0&1
  	\end{array}
  	\right)  
  	\left(
  	\begin{array}{c}
  		v \\ \hline 1
  	\end{array}
  	\right)
  	\\
  	=
  	\left(
  	\begin{array}{c}
  		\exp(B_{\alpha}t)(v)+
  		F_{B_{\alpha}}(t)(d_{\alpha}) \\ \hline 1
  	\end{array}
  	\right) 
  	=\left(
  	\begin{array}{c}
  		v+F_{B_{\alpha}}(t)
  		\left(B_{\alpha}v+d_{\alpha}\right) \\ \hline 1
  	\end{array}
  	\right) 
  \end{eqnarray*}
Hence $t \in \R \longmapsto c(t)\in \R^n$
is the orbit of a one-parameter subgroup of 
the affine group applied to the vector $v.$   
\end{remark}
\section{Smooth curves invariant under $M_{\alpha}$}
For a smooth curve $c: \R \longrightarrow \R^n$
and a parameter $\alpha \in (0,1)$
we define the one-parameter family
$\tilde{c}_s:\R\longrightarrow \R^n, s \in \R$ 
by Equation~\eqref{eq:cstilde}.
And we call a smooth curve $c: \R \longrightarrow \R^n$ a
{\em soliton} of the mapping $M_{\alpha}$
(resp. affinely invariant under $M_{\alpha}$)
if there is $\epsilon >0$ and a
smooth map $ \in (-\epsilon,\epsilon)
\longrightarrow (A,b) \in \affn$ such that
\begin{equation}
\label{eq:soliton-curve}
\tilde{c}_s(t)=
(1-\alpha)c(t)+\alpha c(t+s)=
A(s)(c(t))+b(s).
\end{equation}
Then we obtain as an analogue of 
\cite[Thm.1]{Rademacher2017}:
\begin{theorem}
	\label{thm:curve-soliton}
	Let $c:\R \longrightarrow \R^n$ be a
	soliton of the mapping $M_{\alpha}$
	satisfying Equation~\eqref{eq:soliton-curve}.
	Assume in addition that for some $t_0 \in \R$
	the vectors $\dot{c}(t_0),\ddot{c}(t_0),
	\ldots,c^{(n)}(t_0)$ are linearly 
	independent. 
	
	Then the curve $c$ is the unique solution of the
	differential equation
	\begin{equation*}
	\dot{c}(t)=Bc(t)+d
	\end{equation*}
	for $B=\alpha^{-1}A'(0), d=\alpha^{-1}b'(0)$
	with initial condition $v=c(0).$
	
	And $A(s)=(1-\alpha)\Identity +
	\alpha \exp(Bs), b(s)=\alpha F_B(s)(d).$
	
	Hence the curve $c(t)$ is the orbit 
	of a one-parameter subgroup
	\begin{equation*}
	t \in \R\longmapsto
	B(t):=\exp \left(
	\left(
	\begin{array}{c|c}
	B& d\\ \hline
	0&0
	\end{array}
	\right)
	t \right)
=\left(\exp(Bt), F_B(t)(d)\right)\in \affn
	\end{equation*}	
	of the affine group, i.e. 
	\ben
	c(t)=B(t)\left(\begin{array}{c}
		v \\ \hline 1
	\end{array}\right)
	=v+F_B(t)\left(Bv+d\right)\,,
	\een
	cf. Remark~\ref{rem:one-parameter-subgroup}.
\end{theorem}
\begin{remark}
	\rm
	\label{rem:dimension-two}

For an affine map $(A,b)\in \gl(n),b\in \R^n$
the linear isomorphism $A$ is called
the \emph{linear part}.
For $n=2$ we discuss the possible normal forms
of $A \in \gl(2)$ resp. the normal forms
of the one-parameter subgroup $\exp(tB)$ and
of the one-parameter family $A(s)=
(1-\alpha)+\Identity+\exp(Bs)$
introduced in  
Theorem~\ref{thm:curve-soliton}.
This will be used in 
Section~\ref{sec:affine-curvature}.
\begin{enumerate}
	\item $A=\begin{pmatrix}
	\lambda & 0\\0& \mu
	\end{pmatrix}$ for $\lambda,\mu\in \R-\{0\},$ i.e.
	$A$ is diagonalizable (over $\R$), then $A$ is called
	\emph{scaling,} for $\lambda=\mu$ it is called
	\emph{homothety}.
	For an endomorphism $B$ which is diagonalizable over
	$\R$ the one-parameter subgroup
	$B(t)=\exp(Bt)$
as well as the one-parameter family
$A(s)=(1-\alpha)\Identity+\alpha \exp(Bs)$ consists of
scalings.
		\item $A=\begin{pmatrix}
	a&-b\\b&a
	\end{pmatrix}$
	for $a,b\in \R, b\not=0,$ i.e. $A$ has no
	real eigenvalues. Then $A$ is called a
	\emph{similarity,} i.e. a composition of a
	\emph{rotation} and a 
	\emph{homothety.} 
	For an endomorphism $B$ with no real eigenvalues 
	 the one-parameter
	subgroup
	$B(t)=\exp(Bt), t\not=0$
	as well as the one-parameter family
	$A(s)=(1-\alpha)\Identity+\alpha \exp(Bs),
	s\not=0$ consist of
	affine mappings without real eigenvalues, i.e. compositions
	of non-trivial rotations and homotheties.
	\item $A=\begin{pmatrix}
	1&1\\0 &1
	\end{pmatrix}$ is called
	\emph{shear transformation.} Hence the matrix
	$A$ has only one eigenvalue $1$ and is not
	diagonalizable.
	If $B$ is of the form
	$B=\begin{pmatrix} 0&1\\0&0
	\end{pmatrix},$ i.e. $B$ is nilpotent, then
	the one-parameter
	subgroup
	$B(t)=\exp(Bt), t\not=0$
	as well as the one-parameter family
	$A(s)=(1-\alpha)\Identity+\alpha \exp(Bs),
	s\not=0$ consist of
	shear transformations.
	\item $A=\begin{pmatrix}
	\lambda &1\\0 &\lambda
	\end{pmatrix}$
	with $\lambda \in \R-\{0,1\}.$
	Then $A$ is invertible with only one eigenvalue 
	$\lambda\not=1$
	and not diagonalizable. This linear map is
	a composition of a homothey and a shear transformation.
	The one-parameter
	subgroup
	$B(t)=\exp(Bt), t\not=0$
	as well as the one-parameter family
	$A(s)=(1-\alpha)\Identity+\alpha \exp(Bs),
	s\not=0$ consist of linear mappings with only
	one eigenvalue different from $1$
	which are not diagonalizable.
	Hence they are compositions of non-trivial
	homotheties and shear transformations, too.
	\end{enumerate}

\end{remark}
We use the following convention: For a one-parameter family
$s \mapsto c_s$ of curves or a one-parameter family $s \mapsto A(s),
s \mapsto b(s)$ of affine maps we denote the differentiation
with respect to the parameter $s$ by $'.$ On the other hand
we use for the differentiation with respect to the curve
parameter $t$ of the curves $t \mapsto c(t), t \mapsto c_s(t)$ 
the notation $\dot{c}, \dot{c_s}.$
\begin{proof}
The proof is similar to the Proof of 
Theorem~\cite[Thm.1]{Rademacher2017}:
Let 
\begin{equation}
\label{eq:cs2}
\cas(t)=\Aas c(t)+\bas=
(1-\alpha)
 c(t)+\alpha c(t+s)
\,.
\end{equation}
For $s=0$ we obtain
$
c(t)=c_{0}(t)=A(0)c(t)+b(0)
$
for all $t \in \R,$ resp. 
$
\left(A(0)-\Identity\right)(c(t))=-b(0)
$
for all $t.$ We conclude that
\be
\label{eq:a02}
\left(A(0)-\Identity\right)\left(c^{(k)}(t)\right)=0
\ee
for all $k\ge 1.$
Since for some $t_0$ the vectors
$\dot{c}(t_0),\ddot{c}(t_0),\ldots,c^{(n)}(t_0)$
are linearly independent by assumption we conclude
from 
Equation~\eqref{eq:a02}:
$A(0)=\Identity,b(0)=0.$
Equation~\eqref{eq:cs2} implies for $k\ge 1:$
\ben
A(s)c^{(k)} (t)=
(1-\alpha)
c^{(k)}(t)+\alpha c^{(k)}(t+s)
\een
and hence
\ben
A'(s)c^{(k)}(t)=\alpha
c^{(k+1)}(t+s)\,.
\een
We conclude from Equation~\eqref{eq:cs2}:
\begin{eqnarray*}
\frac{\partial \cas (t)}{\partial s}&=&A'(s)c(t)+b'(s)\\
&=&
\frac{\partial \cas (t)}{\partial t}-
(1-\alpha) \, \dot{c}(t)=
\left(A(s)-(1-\alpha)\Identity\right)\dot{c}(t)\,.
\end{eqnarray*}
Since $A(0)=\Identity$
the endomorphisms 
$A(s)+(\alpha-1)\Identity$ are isomorphisms for
all $s \in (0,\epsilon)$ for a sufficiently small
$\epsilon>0\,.$ Hence we obtain for 
$s \in (0,\epsilon):$
\begin{equation}
\label{eq:odecs}
\dot{c}(t)=
\left(A(s)+(\alpha-1)\Identity\right)^{-1}A'(s)c(t)
+ 
\left(
A(s)+(\alpha-1)\Identity
\right)^{-1}b'(s)\,.
\end{equation}
Differentiating with respect to $s:$
\ben
\left(
\left(
A(s)+(\alpha-1)\Identity
\right)^{-1}
A'(s)\right)'\left(c(t)\right)
+
\left(
\left(
A(s)+(\alpha-1)\Identity
\right)^{-1}
b'(s)\right)'
=0
\een
and differentiating with respect to $t:$
\ben
\left(
\left(
A(s)+(\alpha-1)\Identity
\right)^{-1}
A'(s)\right)'c^{(k)}(t)=0\,;
\,k=1,2,\ldots,n\,.
\een
By assumption the vectors 
$\dot{c}(t_0),\ddot{c}(t_0),\ldots,c^{(n)}(t_0)$ 
are linearly independent. Therefore we obtain
$\left(
\left(
A(s)+(\alpha-1)\Identity
\right)^{-1}
A'(s)\right)'=0\,.$ Let
$B=\alpha^{-1}A'(0)
, d=\alpha^{-1} b'(0).$ Then we conclude
\begin{equation}
\label{eq:asbs}
A'(s)=
\left(A(s)+(\alpha-1)\Identity\right)
B
\,;\,
b'(s)=
\left(A(s)+(\alpha-1)\Identity\right)
(d)
\,,
\end{equation}
We obtain from Equation~\eqref{eq:odecs}:
\ben
\dot{c}(t)=B c(t)+d\,.
\een
Equation~\eqref{eq:asbs} with $A(0)=\Identity$ implies
$A(s)=(1-\alpha)\Identity +\alpha \exp(Bs).$
And we obtain
$b'(s)=\alpha \exp(Bs)(d)=\alpha F_B'(s)(d).$ Hence 
$b(s)=\alpha F_B(s)(d)$
since $b(0)=0.$ 
\end{proof}
As a consequence we obtain the following
\begin{theorem}
	\label{thm:soliton-orbits}
For a $(n,n)$-matrix $B$ and a vector $d$ any solution of
the inhomogeneous linear differential equation
$\dot{c}(t)=B c(t)+d$ with constant coefficients
is a soliton of the mapping $M_{\alpha}.$ 
These solitons are orbits of a one-parameter subgroup
of the affine group, i.e. they are of the form
given in
Equation~\eqref{eq:ct}.
\end{theorem}
\begin{proof}
Any solution of the equation $\dot{c}(t)=B c(t)+d$
has the form
\ben
c(t)=v+F_B(t)(Bv+d)
\een
with $v=c(0),$ 
cf. Proposition~\ref{pro:fb}. Then
with $A(s)=(1-\alpha)\Identity +\alpha \exp(Bs)$
and $b(s)=\alpha F_B(s)(d)$
we conclude from
Equation~\eqref{eq:fb-derivative}
and Equation~\eqref{eq:fb-functional-eq}:
\begin{eqnarray*}
\tilde{c}_s(t)&= &(1-\alpha)c(t)+\alpha c(t+s)\\
&=&v+(1-\alpha)F_B(t)(Bv+d)
+\alpha F_B(t+s)(Bv+d)\\
&=& v+(1-\alpha)(c(t)-v)+
\alpha \left(F_B(s)+ \exp(Bs) F_B(t)\right)
(Bv+d)\\
&=&(1-\alpha) c(t)+\alpha \exp (Bs)(v)+
\alpha F_B(s)(d) + \alpha\exp (Bs)(c(t)-v)\\
&=&
\left(\left(1-\alpha\right)\Identity +\alpha\exp(Bs)\right)(c(t))
+\alpha F_B(s)(d)
\\
&=&
A(s)c(t)+b(s)\,.
\end{eqnarray*}
Hence $c$ is a soliton of the mapping 
$M_{\alpha},$ cf.
Equation~\eqref{eq:soliton-curve}.
\end{proof}

The curves $c$ invariant under the
process $T$ considered in
\cite{Rademacher2017} define a 
class of curves containing the orbits of
the one-parameter subgroups of the affine group. 
They are solutions of the second
order differential equation $\ddot{c}=Bc +d$
which on the other hand can be reduced to a system
of first order differential equations,
cf.~\cite[Rem.1]{Rademacher2017}.
\section{Curves with constant affine curvature}
\label{sec:affine-curvature}
The orbits of one-parameter subgroups of the affine
group $\aff(2)$ acting on $\R^2$ can also be characterized as curves
of constant 
general-affine curvature
parametrized proportional to
general-affine arc length
unless they are parametrizations of a parabola,
an ellipse or a hyperbola.
This will be discussed
in this section. The one-parameter subgroups are determined
by an endomorphism $B$ and a vector $d.$ We describe
in Proposition~\ref{pro:affine-dim2}
how the general-affine curvature can
be expressed in terms of the
matrix $B.$

For certain subgroups of the affine group $\aff(2)$
one can introduce a corresponding \emph{curvature} and 
\emph{arc length.} One should  be aware that sometimes
in the literature the curvature related to the 
equi-affine subgroup $S\aff(2)$ generated by 
the special linear group ${\rm SL}(2)$
of
linear maps of determinant one and the translations
is also called {\em affine curvature.}
We distinguish in the following between the
\emph{equi-affine curvature} 
$\kea$ and the \emph{general-affine curvature} $\ka$
as well as between the \emph{equi-affine
length parameter} $\sea$ and the 
\emph{general-affine length parameter} $\sa.$

We recall the definition of the
equi-affine and general-affine curvature of a 
smooth plane curve $c:I \longrightarrow \R^2$ 
with $\det(\dot{c}(t)  \, \ddot{c}(t))=
|\dot{c}(t) \, \ddot{c}(t)|\not=0$ for all $t\in I.$

By eventually changing the orientation of the curve
we can assume
$|\dot{c}(t) \, \ddot{c}(t)|>0$ for all $t\in I.$
A reference is the book
by P. and A.Schirokow~\cite[\S10]{Schirokov1962}
or the recent article
by Kobayashi and Sasaki~\cite{KS2019}.
Then 
$\sea(t):=
\int \left|\dot{c}(t) \enspace 
\ddot{c}(t)\right|^{1/3}\,dt$
is called {\em equi-affine arc length.}
We denote by $t=t(\sea)$ the inverse function,
then $\tilde{c}(\sea)=c(t(\sea))$
is the parametrization by 
{\em equi-affine arc length.}
Then $\tilde{c}'''(\sea), \tilde{c}'(\sea)$ are linearly
dependent and the {\em equi-affine curvature}
$\kea(s)$ is defined by
\begin{equation*}
\tilde{c}'''(\sea)=-\kea(\sea)\, \tilde{c}'(\sea)
\end{equation*}
resp.
\begin{equation*}
\kea(s)=\left|\tilde{c}''(\sea)\enspace 
\tilde{c}'''(\sea)\right|\,.
\end{equation*}

Assume that $c=c(\sea), \sea\in I$ is a 
smooth curve parametrized
by equi-affine arc length for which the sign
$ \epsilon=\sign( \kea(s))\in  \{0, \pm 1\}$
 of the 
equi-affine curvature is constant. 
If $\epsilon=0$ then the curve is 
up to an affine transformation
a parabola $(t,t^2).$
Now assume $\epsilon \not=0$ and let
$\Kea=|\kea|=\epsilon \kea.$
Then the {\em general-affine arc length}
$\sa=\sa(\sea)$ is defined by 
\begin{equation}
\sa=\int \sqrt{\Kea(\sea)}\,d\sea\,.
\end{equation} 
We call a curve 
$c=c(t)$ {\em parametrized proportional to 
general-affine arc length}
if $t=\lambda_1\sa +\lambda_2$ for 
$\lambda_1,\lambda_2 \in \R$
with $\lambda_1\not=0.$
The {\em general-affine curvature}
$\ka=\ka(s)$ is defined by
\begin{equation}
\label{eq:curv-affine}
\ka(s)=\Kea '(s)
 \Kea (s)^{-3/2}
=
-2 \left(
\Kea^{-1/2} (s)
\right)'\,.
\end{equation}
If the 
general-affine curvature $\ka$ (up to sign) and the 
sign $\epsilon$ is given
with respect to the equi-affine arc length
parametrization, then the equi\--affine curvature
$\kea=\kea(\sea)$ is determined 
up to a constant by
Equation~\eqref{eq:curv-affine}.
Hence the curve is determined up to an affine
transformation. 
The invariant $\ka$ already occurs in
Blaschke's book~
\cite[\S 10, p.24]{Blaschke1923}.
Curves of constant general-affine curvature are orbits of
a one-parameter subgroup of the affine group.
These curves already were discussed
by Klein and Lie~\cite{Klein1871} under the
name \emph{$W$-curves.} 
\begin{proposition}
\label{pro:affine-dim2}
For a non-zero matrix  $B \in M_{\R}(2,2)$ and vectors
$d, v\in \R^2$ where $Bv+d$ is not an eigenvector
of $B$
let 
$c:\R \longrightarrow \R^2$ be the solution of the 
differential equation $c'(t)=B c(t)+d; c(0)=v,$
i.e.
$c(t)=v+F_B(t)(Bv+d)=\exp(tB)(v)+F_B(t)(d).$
We assume that 
$\beta=\left|c'(0)\enspace c''(0)\right|^{1/3}=
\left|Bv+d \enspace B(Bv+d)\right|^{1/3}
>0.$
Define 
\begin{equation}
k=k(B)=-2+9\det(B)/\trace^2(B)
\enspace;\enspace
K=K(B)=|k(B)|^{-1/2}.
\end{equation}

\medskip

(a) If $\trace(B)=0$
then the curve is parametrized proportional
to equi-affine arc length and
the equi-affine curvature is constant
$\kea=\det(B)/\beta^2$ and $
\epsilon=\sign(\det(B)),$ 
the curve is a parabola, if $\epsilon=0,$
an ellipse, if $\epsilon >0,$
or a hyperbola, if $\epsilon<0,$
cf. Remark~\ref{rem:keaconst}.

\medskip

(b) If $\trace(B)\not=0$ then
we can choose a 
parametrization 
by equi-affine arc length $\sea$
such that the 
equi-affine curvature
$\kea$
is given
 by: 
\begin{equation}
\label{eq:eq-affine-curv}
 \kea(\sea)=
 k(B){\sea}^{-2}.
\end{equation}
If $k(B)=0$ the curve has vanishing equi-affine
curvature and is a parametrization of a parabola,
cf. the Remark~\ref{rem:keaconst}.
 If $k(B)\not=0$ then the 
 general-affine curvature
 is defined and constant:
 \begin{equation}
 \label{eq:ka}
 \ka(\sea)=-2 K(B).
 \end{equation}
 Up to an additive constant the
 general-affine arc length
 parameter $\sa$ is given by:
 $$
 \sa=\frac{\trace B}{3 K(B)}\,t\,.
 $$
 Hence the curve $c(t)$ is
 parametrized proportional
 to general-affine arc length.
 \end{proposition}
\begin{remark}
\label{rem:keaconst}
\rm
It is well-known that
the curves of constant equi-affine curvature
are parabola, hyperbola or ellipses,
cf.\cite[\S 7]{Blaschke1923}.
For $\kea=0$ we obtain a parabola:
$c(t)=c(0)+c'(0)\sea+c''(0)\sea^2/2,$
for $\kea>0$ the ellipse
$c(\sea)=\left(a \cos(\sqrt{\kea}\sea),
b \sin(\sqrt{\kea}\sea)\right)$
with $\kea=(ab)^{-2/3}$ and for
$\kea<0$ the hyperbola
$c(\sea)=\left(a \cosh(\sqrt{-\kea}\sea),
b \sinh(\sqrt{-\kea}\sea)\right)$
with $\kea=-(ab)^{-2/3}.$
Here $a,b>0.$

\end{remark}

\begin{proof}
Following Proposition~\ref{pro:fb} 
we obtain as solution of the 
differential equation:
$c(t)=v+F_B(t)(Bv+d),$ 
hence for the derivatives:
$c^{(k)}(t)=B^{k-1} \exp(tB)(Bv+d).$
Then :
\begin{eqnarray*}
 \left| \dot{c}(t) \, 
 \ddot{c}(t)\right|&=&\left|\exp(Bt)\right|
 \left|bv+d \enspace B(Bv+d)\right|
 \\
 &=&
 \exp\left(\trace(B)t\right)
 \left|Bv+d \enspace B(Bv+d)\right|.
\end{eqnarray*}
Let $\beta=
\left(\left| 
Bv+d \enspace B(Bv+d)
\right|\right)^{1/3}$
and $\tau=\trace(B).$ Then
\ben
|\dot{c}(t) \, \ddot{c}(t)|=\beta^3\exp(\tau t).
\een

\medskip

(a) If $\tau=0$ then $\sea=t\beta,$
i.e. the curve is parametrized proportional
to equi-affine arc length
and 
\ben
\tilde{c}(\sea)=c(t(\sea))=c(\sea/\beta)=
v+F_B(\sea/\beta)(Bv+d).
\een
Then
\begin{eqnarray*}
\tilde{c}'(\sea)&=& 
\beta^{-1}\exp(B\sea/\beta)(Bv+d)\\
\tilde{c}'''(\sea)&=&
\beta^{-3}B^2\exp(B\sea/\beta)(Bv+d)
=-\det(B)\beta^{-2}\tilde{c}'(\sea)\,.
\end{eqnarray*}
Here we use that by Cayley-Hamilton
$B^2-\tau B=B^2=-\det (B) \cdot \Identity.$ Hence we obtain
$\kea(s)=\det(B)/\beta^2$ and 
$\epsilon = \sign (\det(B)).$
Then the claim follows from
Remark~\ref{rem:keaconst}.

\medskip

(b) Assume $\tau\not=0.$
Then the equi-affine arc length $\sea=\sea(t)$ is given by
\be
\label{eq:seatau}
\sea(t)=\beta
\int\exp(\tau t/3)\,dt= 
\frac{3\beta}{\tau}\exp(\tau t/3)\,.
\ee
Hence the equi-affine arc length parametrization
of $c$ is given by
\begin{equation*}
\tilde{c}(\sea)=v
+F_B\left(\frac{3}{\tau}
\ln
\left(
\frac{\tau}{3\beta}\sea 
\right)
\right)(Bv+d).
\end{equation*}
Then we can express the derivatives:
\begin{eqnarray*}
\tilde{c}'(\sea)&=&
\frac{3}{\tau}\, \frac{1}{\sea}\exp\left(\frac{3}{\tau}
B\ln\left(\frac{\tau}{3\beta}\sea\right)
\right)(Bv+d)\\
\tilde{c}''(\sea)&=&
\left(\frac{3}{\tau}B-\Identity\right)
\frac{1}{\sea}\tilde{c}'(\sea)\\
\tilde{c}'''(\sea)&=&
\left(\frac{3}{\tau}B-\Identity\right)
\left(\frac{3}{\tau}B-2\Identity\right)
\frac{1}{\sea^2}\tilde{c}'(\sea)\\
&=&-\left(\frac{9\det(B)}{\tau^2}-2
\right)\frac{1}{\sea^2}\tilde{c}'(\sea)\,.
\end{eqnarray*}
Here we used that by Cayley-Hamilton
$B^2-\tau B=-\det B\cdot\Identity.$
Hence we obtain for the equi-affine
curvature
\begin{equation}
\label{eq:eqaffine-curv}
\kea(\sea)=\frac{k(B)}{\sea^2}.
\end{equation}
Then $\epsilon=\sign \, k(B)$ and
for $k(B)\not=0$ we obtain
from Equation~\eqref{eq:curv-affine}
and Equation~\eqref{eq:eqaffine-curv}:
\begin{equation*}
 \ka(\sea)=-
\frac{2 }{K(B)}
 \end{equation*}
And for the general-affine arc length
we obtain
\ben
\sa=
\ln (|\sea|)/K(B)
\,,
\een
resp. up to an
additive constant:
\ben
\sa=
\frac{\tau t}{3 K(B)}
\een
using Equation~\eqref{eq:seatau}.

The parametrization by general-affine arc length
is given by 
\ben
\label{eq:parametrization-affine}
c^*(\sa)=v+
F_B\left(3 K(B) \sa/\tau
\right)(Bv+d)
\,.
\een
\end{proof}

\begin{example}
	\label{exa:n=2}
\rm

	Depending on the real Jordan normal forms of the endomorphism
	$B$ we investigate the solitons $c(t),$ their
	special and general affine curvature. The normal forms
	of the corresponding one-parameter subgrpoup 
	$B(t)=\exp(Bt)$ as well
	of the one-parameter family $A(s)=(1-\alpha)\Identity+\exp(Bs)$
	follow from Remark~\ref{rem:dimension-two}.
	Since $c(\mu t)=\exp(\mu B t)$ the multiplication
	of $B$ with a non-zero real $\mu$ corresponds to
	a linear reparametrization of the curve. 
	If $B$ has a non-zero real
	eigenvalue we can assume without loss of generality
	that it is $1$
	and in the case of a non-real eigenvalue we can assume that
	it has modulus $1.$
		\begin{itemize}
\item[(a)]
Let $B=
\left(
\begin{array}{cc}
1 & 0\\
0 & \lambda
\end{array}
\right), d=(0,0), c(0)=(1,1)$
and $\lambda\not=0,1.$
Then $\beta=(\lambda(1-\lambda))^{1/3}\not=0,
\trace B=1+\lambda$
and $c(t)=\left(\exp( t), 
\exp(\lambda t)\right).$
Up to parametrization we have
$c(u)=\left(u, u^{\lambda}\right).$

If $\lambda=-1$ then $c$ is a
parametrization of a hyperbola,
$\trace B=0$ and $\kea=- 2^{-2/3},$
cf. Remark~\ref{rem:keaconst}.

If $\lambda\not=-1$ 
we obtain for the equi-affine curvature
with respect to a equi-affine parametrization
$\sea$
from Equation~\eqref{eq:eq-affine-curv}:
\begin{eqnarray*}
\kea(\sea)=\left(9\frac{\det B}{\trace^2B}-2\right)\frac{1}{\sea^2}=
-\frac{(\lambda-2)( 2\lambda-1)}{
(\lambda+1)^2}\frac{1}{\sea^2}\,.
\end{eqnarray*}
For $\lambda=1/2,2$ we obtain
a parametrization of a parabola
with vanishing equi-affine curvature,
cf. Remark~\ref{rem:keaconst}.
Now we assume $\lambda\not=1/2,2.$
Hence $\epsilon=1$ if and only if $1/2<\lambda <2.$
The affine curvature $\ka$ is constant:
\begin{equation*}
\ka=-
2\,\frac{
|\lambda +1|}{
\sqrt{|(\lambda-2)(2\lambda-1)|}
}\,,
\end{equation*}
cf. \cite[Ex.2.14]{KS2019}.
We have $\epsilon=1$ if and only if
$1/2<\lambda<2,$ then $\kga \in (-\infty,-4).$
And $\epsilon=-1$ if and only if $\lambda<1/2, \lambda\not=0$ or
$\lambda>2,$ then
$\kga\in (-\infty,-\sqrt{2})\cup (-\sqrt{2},0).$

Hence in this case the corresponding
one-parameter subgroup 
$$B(t)=
\begin{pmatrix}
\exp(t)&0\\
0&\exp(\lambda t)
\end{pmatrix}$$ as well as
the one-parameter family 
$$A(s)=\begin{pmatrix}
1-\alpha+\alpha \exp(s)&0\\
0&1-\alpha + \alpha \exp(\lambda s)
\end{pmatrix}$$
 consist of 
scalings.
\item[(b)]
$B=
\left(
\begin{array}{cc}
0 & 0\\
0 & 1
\end{array}
\right)$
and $d=(1,0).$
Then the solution of the Equation
$\dot{c}(t)=B c(t)+d$ with $c(0)=(0,1)$
is of the form
$c(t)=(t,\exp(t)).$
Then we obtain $\epsilon=-1$
and 
$\ka=-\sqrt{2}.$
The corresponding
one-parameter subgroup $B(t)$ as well as
the one-parameter family $A(s)$ consist of 
scalings, the affine transformation
$(A(s),b(s)$ is given by
$(A(s),b(s))=
\left(
\begin{pmatrix} 1&0\\0&1-\alpha+\alpha \exp(s)
\end{pmatrix}, 
\alpha
\begin{pmatrix} s\\0
\end{pmatrix}
\right),$
i.e. a composition of scalings and
translations.
\begin{figure}
	\includegraphics[width=0.6\textwidth]{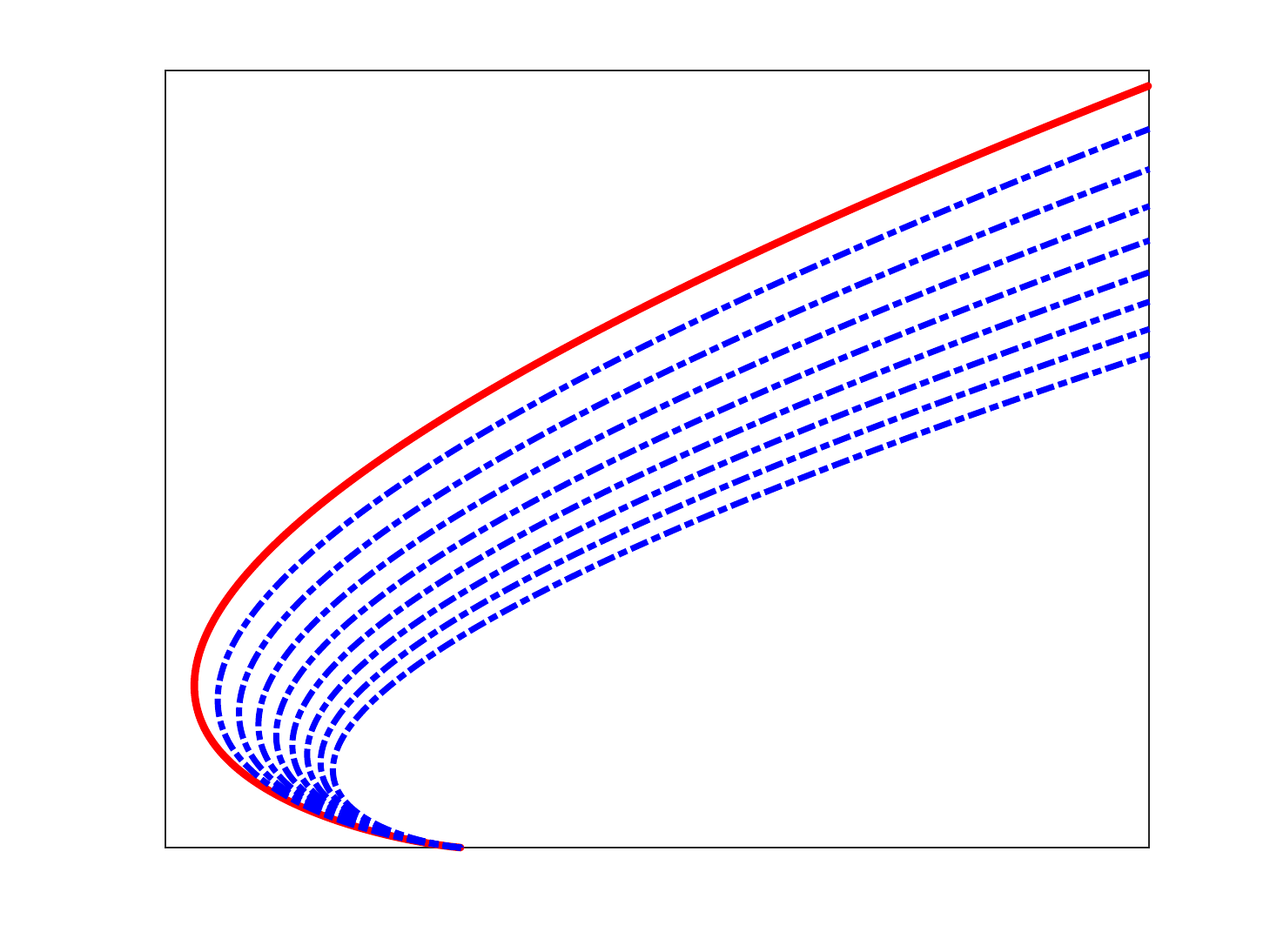} 
	\caption{The soliton $c(t)=( (t+1)\exp(t), \exp(t))$} with the family $c_s(t)=A(s)c(t).$
	\label{fig:two}
\end{figure}

\item[(c)]
If $B=
\left(
\begin{array}{cc}
1 & 1\\
0 & 1
\end{array}
\right), d=(0,0), c(0)=(1,1)$ then $c(t)=\left((t+1)\exp( t), 
\exp( t)\right)$
i.e. up to an affine transformation
and a reparametrization
the curve is 
of the form
$c(u)=(u, u\ln(u)).$
Then $\epsilon=1$ and
$\ka=-4.$
The corresponding
one-parameter subgroup $B(t)$ as well as
the one-parameter family $A(s)$ consist of 
compositions of a homothety and a shear transformations;
$$
B(t)=\exp(t) \cdot
\begin{pmatrix}
1&t\\0&1
\end{pmatrix}\,;\,
A(s)=
\begin{pmatrix}
1-\alpha+\alpha \exp(s)&\alpha s \exp(s)\\
0&1-\alpha+\alpha\exp(s)
\end{pmatrix},
$$ cf. Figure~\ref{fig:two}.
\item[(d)]
If $B=
\left(
\begin{array}{rr}
a & -b\\
b & a
\end{array}
\right)$ 
with $b\not=0, a^2+b^2=1, d=0, c(0)=(1,0)$
then 
$c(t)=\exp(at)\left(\cos(b t),
\sin(b t)\right)$.
For $a=0,$ this is a circle with
$\kea=1.$ Now we assume
$a\not=0:$
Then $\epsilon=\sign (9 \det(B)-2 \trace^2(B))=    
\sign(a^2+9b^2)=1$ and
we obtain for the general-affine
curvature
$\ka=-4 |a|
/\sqrt{a^2+9b^2}=-4|a|/\sqrt{9-8a^2}\,,$
i.e. $\ka \in (-4,0).$
The corresponding
one-parameter subgroup $B(t)$ as well as
the one-parameter family $A(s)$ consist
of similarities. 
\item[(e)] If $B=
\left(
\begin{array}{cc}
0 & 1\\
0 & 0
\end{array}
\right)$
one can choose $c(0)=(0,0),d=(0,1)$
and obtain $c(t)=(t^2/2,t),$ i.e. a
parabola. In this case the 
one parameter subgroup 
$B(t)=\exp(tB)$ consists
of shear transformations. 
The one-parameter family
$(A(s),b(s))=
\left(
\begin{pmatrix} 1&\alpha s\\0&1
\end{pmatrix}, 
\alpha
\begin{pmatrix} s^2/2\\s
\end{pmatrix}
\right)$
consists of a composition of a shear transformation
and a translation,
cf. Figure~\ref{fig:one}.
For the \emph{affine curve shortening flow}
the parabola is a \emph{translational soliton}.
Therefore it is also called 
the \emph{affine analogue of the grim reaper,}
	cf. \cite[p.192]{Calabi1996}.
	For the curve shortening process
	$T$ defined by Equation~\eqref{eq:rademacher1}
	the parabola is also a translational soliton,
	cf. \cite[Sec.5, Case (5)]{Rademacher2017}.

\end{itemize}
\end{example}
Note that the parabola occurs twice, 
in Case (a) it occurs with the parametrization
$c(t)=(\exp(t),\exp(2t)),$ 
in Case (e) it occurs with a parametrization
proportional to equi-affine arc length.
Summarizing we obtain from 
Theorem~\ref{thm:curve-soliton} and
Theorem~\ref{thm:soliton-orbits} together with
Proposition~\ref{pro:affine-dim2}
resp. Example~\ref{exa:n=2}
the following
\begin{theorem}
	Let $c:\R\longrightarrow
	\R^2$ be a smooth curve for which
	$\dot{c}(0),\ddot{c}(0)$ are linearly independent.
	Then $c$ is a soliton
of the mappings
$M_{\alpha},\alpha\in (0,1),$ in particular
of the midpoints mapping $M=M_{1/2},$
if it is a curve of constant
equi-affine curvature  
parametrized proportional to
equi-affine arc length, 
or a parabola with the parametrization
$c(t)=(\exp(t),\exp(2t))$
up to an affine transformation,
or if it is a
curve of constant general-affine curvature 
parametrized proportional to general-affine
arc length.
\end{theorem}

\end{document}